\documentclass[12pt,twoside]{amsart}
\usepackage{amsmath}
\usepackage{epsfig}
\usepackage{amssymb}
\usepackage{xcolor}

\usepackage{hyperref}

\newtheorem{theorem}{Theorem}[section]
\newtheorem{lemma}[theorem]{Lemma}

\newtheorem{conjecture}[theorem]{Conjecture}
\newtheorem{problem}[theorem]{Problem}

\theoremstyle{definition}

\theoremstyle{remark}

\numberwithin{equation}{section}

\newcommand{\cB}{{\mathcal B}}
\newcommand{\cC}{{\mathcal C}}
\newcommand{\cD}{{\mathcal D}}
\newcommand{\cG}{{\mathcal G}}

\newcommand{\subproof}{\begin{proof}[Subproof]}

\newcommand{\del}{\backslash}
\newcommand{\con}{/}

\newcommand{\bZ}{\mathbb Z}
\newcommand{\bF}{\mathbb F}

\newcommand{\bFp}{{\bF}_{\rm prime}}
\newcommand{\bFmult}{{\bF}^\times}

\newcommand{\cR}{\mathcal{R}}

\newcommand{\cM}{\mathcal{M}}

\DeclareMathOperator{\pert}{pert}
\DeclareMathOperator{\dist}{dist}
\DeclareMathOperator{\rs}{RowSpace}
\DeclareMathOperator{\rank}{rank}
\DeclareMathOperator{\PG}{PG}

\begin{document}

\sloppy

\title[Highly Connected Matroids]{The Highly Connected Matroids in Minor-closed Classes}

\author[Geelen]{Jim Geelen}
\address{Department of Combinatorics and Optimization,
University of Waterloo, Waterloo, Canada} 
\thanks{ This research was partially supported by grants from the
Office of Naval Research [N00014-10-1-0851],
the Marsden Fund of New Zealand, and NWO (The Netherlands Organisation for Scientific
Research).}
\email{jim.geelen@uwaterloo.ca}

\author[Gerards]{Bert Gerards}
\address{Centrum Wiskunde \& Informatica, Amsterdam, The Netherlands}

\author[Whittle]{Geoff Whittle}

\address{School of Mathematical and Computing Sciences,
Victoria University, Wellington, New Zealand}
\email{geoff.whittle@vuw.ac.nz}

\subjclass[2010]{05B35}
\keywords{matroids, minors, connectivity, girth,
linear codes, ML threshold function, growth rate}
\date{\today}

\begin{abstract}
For any minor-closed class of matroids over a fixed finite field,
we state an exact structural characterization for the sufficiently
connected matroids in the class.  We also state a number of conjectures
that might be approachable using the structural characterization.
\end{abstract}

\maketitle


\noindent {\em This paper is dedicated to James Oxley on the  occasion of his 60th Birthday.}

\section{Introduction}

We have proved a structure theorem for members of any given proper minor-closed class of 
matroids representable over a given finite field. The full statement of the structure theorem
involves a number of technicalities and is somewhat lengthy to state.  
The proof, which will appear in the series of papers on matroid structure
that we are currently writing, is also long and technical and it will be some time before
all of these papers are written. 
We anticipate that the structure theorem will have many applications, but owing to the 
nature of the theorem, it will take considerable effort for others to become
proficient in its use.

In the second half of this paper we state a number of conjectures
that one might be able to approach using the structure theory.
There are quite a few problems in matroid theory that reduce to instances of
arbitrarily high connectivity; for example, problems in coding theory and
problems involving quadratic or exponential growth-rates.
Fortunately, most of the technical issues in the statement 
of the structure theorem evaporate when we consider
matroids that are sufficiently highly connected.
In the first part of the paper we state a simplified version of the
structure theorem for highly connected matroids. 
We then state refinements of the structure theorems that
give exact characterizations of the sufficiently highly connected
matroids in minor-closed classes.
We hope that others who are interested in using our more 
general structure theorem will be able to familiarise themselves with
some of the essential ingredients by first applying this simplified version.

\section{Preliminaries}

We follow the notation of Oxley \cite{oxley} except that we
use $|M|$ to denote the size of the ground set of $M$ and that we will
denote the column matroid of a matrix $A$ by $\widetilde{M}(A)$.
What follows is a discussion of some key notions that are particularly
relevant for a reading of this paper.

\subsection*{Connectivity}
Tutte's definition of $k$-connectivity is a bit restrictive since, for example,
projective geometries fail to be $4$-connected.
For some of our intended applications in coding theory and on growth rates,
{\em vertical} $k$-connectivity is more natural.
We recall that a matroid $M$ is {\em vertically $k$-connected}
if, for each partition $(X,Y)$ with $r(X)+r(Y)-r(M)<k-1$, either 
$X$ or $Y$ is spanning.  Vertical connectivity is geometrically natural and
aligns well with vertex connectivity in graphs.

\subsection*{Represented matroids}

While we typically use the language of matroid theory
to discuss our results, our structure theorems are about matrices,
so it is convenient to have a more formal notion of a
``representation".

For a field $\bF$, an {\em $\bF$-represented matroid} is a pair
$M=(E,U)$ where $U$ is a subspace of $\bF^E$.
For a matrix $A$ over $\bF$ with columns indexed by $E$, we let $M(A)$ denote
$(E,\rs(A))$; we call $A$ a
{\em generator matrix} for $M$ and say that 
$A$ {\em generates} $M$. 
For a represented matroid $M$ we denote by $\widetilde M$ 
the column matroid of a generator matrix of $M$.
We freely carry standard notions such as
circuits, bases, etcetera, over from $\widetilde{M}$ to $M$.

Two $\bF$-represented matroids $(E,U_1)$ and $(E,U_2)$ are
{\em projectively equivalent} if they are the same up to
``conversion of units'', that is if $U_2=\{xD\,:\,x\in U_1\}$
for some nonsingular diagonal matrix $D$.
This means that their generator matrices are
row equivalent up to column scaling.

The matroid operations:  deletion, contraction, and
duality  all have well understood analogues for represented matroids.
For a set $X$ of elements of an
$\bF$-represented matroid $M=(E,U)$, we define 
\begin{eqnarray*}
U|X &=& \{ u|X\, : \, u\in U\}, \\
U\del X &=& U|(E-X), \mbox{ and}\\
 U\con X &=& \{ u\in U\, : \, u|X=0\}\del X.
\end{eqnarray*}
If $X$ and $Y$ are disjoint sets in $E$,
then the matroid obtained from $M$ by {\em deleting} $X$ and {\em contracting} $Y$
is  $M\del X\con Y=(E-(X\cup Y),U\del X\con Y)$.
We call any represented matroid that is projectively equivalent to $M\del X\con Y$
for some $X,Y\subseteq E$ a {\em minor} of $M$.

The {\em dual} of $M=(E,U)$ is defined as $M^*= (E, U^\perp)$,
where $U^{\perp}$ denotes the subspace of $\bF^E$ consisting of all
vectors that are orthogonal to each vector in $U$.

\subsection*{Perturbations}  

We introduce three interrelated operations on a representation:
projection, lifting, and perturbation.
Informally, projection is the operation of extending 
by a set of new elements and then contracting them;
lifting is the  dual operation of coextension and deletion;
and rank-$t$ perturbation  is the operation of
adding a matrix of rank $t$.
Throughout this section $\bF$ denotes a field and
$E$ denotes a finite set.

Let $M_1=(E,U_1)$ and $M_2 = (E,U_2)$ be $\bF$-represented matroids.
If there is a represented matroid $M$ on ground set
$E\cup \{e\}$ such that
$M_1 = M\del e$ and $M_2=M\con e$, then we say that $M_2$ is an {\em elementary
projection} of $M_1$ and that $M_1$  an {\em  elementary lift} of $M_2$. 
We let $\dist(M_1,M_2)$ denote the minimum number of elementary lifts and
elementary projections required in order to transform $M_1$ into $M_2$.

We say that $M_2$ is a {\em rank-$(\le t)$ perturbation}
of $M_1$ if there exist generator matrices $A_1$ for $M_1$
and $A_2$ for $M_2$, with
the same set of row indices, such that $\rank(A_1-A_2) \le t$.
If $t$ is the smallest integer such that $M_2$ is a rank-$(\le t)$
perturbation of $M_1$, then we write $\pert(M_1,M_2)=t$.
The next
result follows from \cite{ggw1}.

\begin{lemma}
\label{dist-pert}
If $M_1$ and $M_2$ are $\bF$-represented matroids on the same ground set, then
\[ \pert(M_1,M_2)\leq \dist(M_1,M_2)\leq 2\pert(M_1,M_2).\]
\end{lemma}

\subsection*{Frame matrices and confinement}

Let $A$ be a matrix over a field $\mathbb F$. Then $A$ is a {\em frame matrix} if 
each column of $A$ has at most two nonzero entries.
A {\em represented frame matroid} is a represented matroid that is
generated by a frame matrix.

We let $\bFmult$ denote the multiplicative group of $\bF$.
Let $\Gamma$ be a subgroup of $\bFmult$.
A {\em $\Gamma$-frame matrix} is a frame matrix $A$ such that:
\begin{itemize}
\item each column of $A$ with one nonzero entry contains a $1$, and
\item each column of $A$ with two nonzero entries contains a $1$
and a distinct entry $-\gamma$ where $\gamma\in \Gamma$.
\end{itemize}
If $A$ is a $\Gamma$-frame matrix, then we call
$M(A)$ an {\em $\bF$-represented $\Gamma$-frame matroid};
the set of all $\bF$-represented $\Gamma$-frame matroids
is denoted by $\cD(\bF,\, \Gamma)$.
We let $\cD(\bF,\, \Gamma)^*$ denote the set of duals of
elements of $\cD(\bF,\, \Gamma)$. 

Let $\bF'$ be a subfield of a field $\bF$ and $M$
be an $\bF$-represented matroid.
Then $M$ is {\em confined to $\bF'$} if there exists
a matrix $A$ over $\bF'$ such that $M$
is projectively equivalent to $M(A)$.

Every finite field $\bF$ has a unique subfield of prime
order; we denote that subfield by $\bFp$.

\section{Structure in minor-closed classes}

We will call a class of matroids or represented matroids {\em minor closed} 
if it is closed under both minors and isomorphism.

We can now state the structure theorem for highly connected matroids
in a proper minor-closed class of matroids
representable over a finite field.

\begin{theorem}
\label{structure1}
Let $\mathbb F$ be a finite field and 
let $\mathcal M$ be a proper minor-closed class of $\mathbb F$-represented matroids.
Then there exist $k,t\in\mathbb Z_+$ such that
each vertically $k$-connected member of $\mathcal M$ is 
a rank-$(\leq t)$ perturbation of an $\mathbb F$-represented matroid
$N$,
such that either
\begin{itemize}
\item[(i)] $N$ is a represented frame matroid, 
\item[(ii)] $N^*$ is a represented frame matroid, or
\item[(iii)] $N$ is confined to a subfield of $\mathbb F$.
\end{itemize}
\end{theorem} 

The outcomes in Theorem~\ref{structure1} are not mutually exclusive,
but the following sequence of results describes how each outcome arises.
The first of these results generalises
a theorem of Mader \cite{mader} that any sufficiently connected graph
has a $K_n$-minor.

\begin{theorem}
\label{structure3}
Let $\mathbb F$ be a finite field and let $n$ be a positive integer.
Then there exists $k\in \bZ_+$ such that
each vertically $k$-connected $\mathbb F$-representable matroid
has an $M(K_n)$- or $M(K_n)^*$-minor.
\end{theorem}

Up to duality we may restrict our attention to matroids that 
contain an $M(K_n)$-minor, where $n$ is arbitrarily large.

\begin{theorem}
\label{structure4}
Let $\mathbb F$ be a finite field and
let $m_0$ be a positive integer. Then there exist
$k, n, t\in \bZ_+$  such that, if $M$ is an
$\mathbb F$-represented matroid such 
that $M$ or $M^*$ is vertically $k$-connected and such
that $\widetilde M$ has an
$M(K_n)$-minor but no $\PG(m_0-1,\bFp)$-minor, then
$M$ is a rank-$(\le t)$ perturbation of an $\bF$-represented frame matroid.
\end{theorem}

Now we may restrict our attention to matroids that contain a $\PG(m_0-1,\bFp)$-minor,
where $m_0$ is arbitrarily large.

\begin{theorem}
\label{structure5}
Let $\mathbb F$ be a finite field and
let $m_1$ be a positive integer. Then there exist
$k, m_0, t\in \bZ_+$ such that,
if $M$ is a vertically $k$-connected
$\mathbb F$-represented matroid such that $\widetilde M$ has
a $\PG(m_0-1,\bFp)$-minor
but no $\PG(m_1-1,\bF)$-minor,
then $M$ is a rank-$(\le t)$ perturbation of
an $\bF$-represented matroid that is confined to a proper subfield of $\bF$.
\end{theorem}

\section{Refinements of the structure}

While Theorem~\ref{structure1} says a lot about 
the structure of highly connected matroids in
minor-closed classes, we can say considerably more.
Indeed, we give a precise structural characterization
of the sufficiently connected matroids.

Before we state these stronger results, we first clarify some notation.
A matrix over a field $\bF$ with rows indexed by  a set $R$
and columns indexed by a set $C$ is an element of $\bF^{R\times C}$;
in particular, matrices do not have ordered rows and columns.
A matrix $A_1\in \bF^{R_1 \times C_1}$ is {\em isomorphic}
to a matrix $A_2\in \bF^{R_2 \times C_2}$ if there
exist bijections $\phi_r:R_1\rightarrow R_2$ and
$\phi_c:C_1\rightarrow C_2$ such that
$A_1[i,j] = A_2[\phi_r(i),\phi_c(j)]$ for each $i\in R_1$
and $j\in C_1$.
A {\em unit vector} is one that contains
exactly one non-zero entry and that entry is $1$.

\subsection*{Confined to a subfield}

We first consider the structure of highly connected matroids that
contain high-rank projective geometries over a proper subfield. In this case,
by Theorem~\ref{structure5}, the  matroids of interest are
a low-rank perturbation from being representable over a subfield.
The structure theorem is somewhat technical; 
to facilitate the description we capture much of the complexity in 
a ``template''.

Let $\mathbb F$ be a finite field. Then a
{\em subfield template} over $\mathbb F$ is a tuple 
$\Phi=(\bF_0,C,D,Y, A_1,A_2, \Delta,\Lambda)$
such that the following hold.
\begin{itemize}
\item[(i)] $\bF_0$ is a subfield of $\bF$.
\item[(ii)] $C$, $D$ and $Y$ are disjoint finite sets.
\item[(iii)] $A_1\in \bF^{D\times C}$ and $A_2\in \bF_0^{D\times Y}$.
\item[(iv)]  $\Lambda$ is a subspace of $\bF_0^D$ and
$\Delta$ is a subspace of $\bF_0^{C\cup Y}$.
\end{itemize}

Let $\Phi = (\bF_0,C,D,Y, A_1,A_2, \Delta,\Lambda)$ be a 
subfield template. Let $E$ be a finite set, let $B\subseteq E$,
and let $A\in \bF^{B\times (E-B)}$.
We say that $A$ {\em conforms} to $\Phi$ if 
the following hold.
\begin{itemize}
\item[(i)] $D\subseteq B$ and $C,Y\subseteq E-B$.
\item[(ii)] $A[D,C] = A_1$, $A[D,Y]=A_2$, and 
all entries of $A$ other than the entries in $A[D,C]$
are contained in the subfield $\bF_0$.
\item[(iii)] Each column of $A[D, E-(B\cup C\cup Y)]$ is contained
in $\Lambda$.
\item[(iv)] Each row of $A[C\cup Y, B-D]$ is contained
in $\Delta$.
\end{itemize}

Figure~\ref{template1} shows the structure of $A$.

\begin{figure}
\begin{center}\begin{tabular}{r|c|c|p{5cm}|}
\multicolumn{1}{c}{}&\multicolumn{1}{c}{$C$}&\multicolumn{1}{c}{$Y$}&\multicolumn{1}{c}{}\\ \cline{2-4}
$D$ &$A_1$&$A_2$&\multicolumn{1}{c|}{\mbox{columns from }$\Lambda$} \\ \cline{2-4}
  &\multicolumn{2}{l|}{}& \\
  &\multicolumn{2}{c|}{rows}& \\
  &\multicolumn{2}{c|}{from}&\multicolumn{1}{c|}{entries from $\bF_0$}\\
  &\multicolumn{2}{c|}{$\Delta$}& \\
  &\multicolumn{2}{l|}{}& \\
\cline{2-4}
\end{tabular}
\end{center}
\caption{The structure of $A$.}
\label{template1}
\end{figure}

Let $M$ be an $\bF$-represented matroid.
We say that $M$ {\em conforms} to $\Phi$
if there is a matrix $A$ that conforms to $\Phi$
such that $M$ is equivalent to $M([I,A])\con C\del D$ 
up to isomorphism and projective transformations.
Let $\cM(\Phi)$ denote the set of $\bF$-represented matroids 
that conform to $\Phi$.
The following theorem is a corollary of the main result in \cite{ggw1};
the actual derivation of Theorem~\ref{refinement1} from \cite{ggw1}
will be given in a later paper.

\begin{theorem}
\label{refinement1}
Let $\mathbb F$ be a finite field and let 
$\mathcal M$ be a minor-closed class of $\mathbb F$-represented matroids.
Then there exist $k,m\in\bZ_+$ and
subfield templates $\Phi_1,\ldots,\Phi_t$ such that
\begin{itemize}
\item
$\cM$ contains each of
the classes $\cM(\Phi_1),\ldots,\cM(\Phi_t)$, and
\item
if $M$ is a simple vertically $k$-connected member of $\mathcal M$
and $\widetilde M$ has a $PG(m-1,\bFp)$-minor, 
then $M$ is member of at least one of the classes
$\cM(\Phi_1),\ldots,\cM(\Phi_t)$.
\end{itemize}
\end{theorem}

\subsection*{Perturbations of represented frame matroids}

We now consider the highly connected matroids
with no $\PG(m-1,\bFp)$-minor that do have an $M(K_n)$-minor,
where $n \gg m$.
By Theorem~\ref{structure4}, these are low-rank perturbations
of represented frame matroids.
We will state an analogue of Theorem~\ref{refinement1}, 
but in this case the structure is a bit more cumbersome.

Let $\mathbb F$ be a finite field. Then a
{\em frame template} over $\mathbb F$ is a tuple 
$\Phi=(\Gamma,C,D,X,Y_0,Y_1, A_1, \Delta,\Lambda)$
such that the following hold.
\begin{itemize}
\item[(i)] $\Gamma$ is a subgroup of the multiplicative
group of $\bF$.
\item[(ii)] $C$, $D$, $X$, $Y_0$ and $Y_1$ are disjoint finite sets.
\item[(iii)] $A_1\in \bF^{(D\cup X)\times (C\cup Y_0\cup Y_1)}$.
\item[(iv)]  $\Lambda$ is a subgroup of the additive group of $\bF^D$ 
and is closed under scaling by elements of $\Gamma$.
\item[(v)]
$\Delta$ is a subgroup of the additive group of $\bF^{C\cup Y_0\cup Y_1}$
and is closed under scaling by elements of $\Gamma$.
\end{itemize}

Let $\Phi = (\Gamma,C,D,X,Y_0,Y_1,  A_1, \Delta,\Lambda)$ be a 
frame template. Let $E$ be a finite set, let $B\subseteq E$,
and let $A'\in \bF^{B\times (E-B)}$.
We say that $A'$ {\em respects} $\Phi$ if 
the following hold.
\begin{itemize}
\item[(i)] $D, X\subseteq B$ and $C,Y_0,Y_1\subseteq E-B$.
\item[(ii)] $A'[D\cup X,C\cup Y_0\cup Y_1] = A_1$ and $A'[X, E-(B\cup C\cup Y_0\cup Y_1)] =0$.
\item[(iii)] There exists a set $Z\subseteq E - (B\cup C\cup Y_0\cup Y_1)$ such that
$A'[D,Z] = 0$, $A'[B-(D\cup X),E-(B\cup C\cup Z\cup Y_0\cup Y_1)]$ is a 
$\Gamma$-frame matrix, and each column of 
$A'[B-(D\cup X), Z]$ is a unit vector.
\item[(iv)]
Each column of $A'[D, E-(B\cup C\cup Y_0\cup Y_1\cup Z)]$ is contained
in $\Lambda$.
\item[(v)] Each row of $A'[B-(D\cup X),C\cup Y_0\cup Y_1]$ is contained
in $\Delta$.
\end{itemize}

Figure~\ref{template2} shows the structure of $A'$.

\begin{figure}
\begin{center}\begin{tabular}{r|p{5.6cm}|c|c|c|c|}
\multicolumn{2}{c}{}& \multicolumn{1}{c}{$Z$}&\multicolumn{1}{c}{$Y_0$}&
\multicolumn{1}{c}{$Y_1$}&\multicolumn{1}{c}{$C$} \\ \cline{2-6}
\begin{tabular}{r}$X$\\ $D$\end{tabular}&
\mbox{}\hspace*{-.25cm}
\begin{tabular}{p{1cm}cp{1cm}}&$0$&\\ \hline &columns from $\Lambda$&\end{tabular}&
0&\multicolumn{3}{c|}{$A_1$}\\ \cline{2-6}
&&&\multicolumn{3}{c|}{}\\
&&&\multicolumn{3}{c|}{rows}\\
&\mbox{}\hfill $\Gamma$-frame matrix \hfill\mbox{}&unit columns&\multicolumn{3}{c|}{from} \\
&&&\multicolumn{3}{c|}{$\Delta$}\\
&&&\multicolumn{3}{c|}{}\\
\cline{2-6}
\end{tabular}
\end{center}
\caption{The structure of $A'$.}
\label{template2}
\end{figure}

Suppose that $A'$ respects $\Phi$ and that $Z$ satisfies (iii) above.
Now suppose that $A\in \bF^{B\times (E-B)}$
satisfies the following conditions.
\begin{itemize}
\item[(i)] $A[B,E-(B\cup Z)] = A'[B,E-(B\cup Z)]$.
\item[(ii)] For each $i\in Z$ there exists $j\in Y_1$ such that
the $i$-th column of $A$ is the sum of the $i$-th and the $j$-th
columns of $A'$.
\end{itemize}
We say that any such matrix $A$ {\em conforms} to $\Phi$.

Let $M$ be an $\bF$-represented matroid.
We say that $M$ {\em conforms} to $\Phi$
if there is a matrix $A$ that conforms to $\Phi$
such that $M$ is equivalent to $M([I,A])\con C\del ((B-X)\cup Y_1)$ 
up to isomorphism and projective transformations.
Let $\cM(\Phi)$ denote the set of $\bF$-represented matroids 
that conform to $\Phi$.  We will prove the following theorem 
in a later paper.

\begin{theorem}
\label{refinement2}
Let $\mathbb F$ be a finite field, let $m$ be a positive integer, and let 
$\mathcal M$ be a minor-closed class of $\mathbb F$-represented matroids.
Then there exist $k\in\bZ_+$ and
frame templates $\Phi_1,\ldots,\Phi_s,\Psi_1,\ldots,\Psi_t$ such that
\begin{itemize}
\item
$\cM$ contains each of the classes
$\cM(\Phi_1),\ldots,\cM(\Phi_s)$,
\item
$\cM$ contains the duals of the matroids in each of the classes
$\cM(\Psi_1),\ldots,\cM(\Psi_t)$, and
\item
if $M$ is a simple vertically $k$-connected member of $\mathcal M$
and $\widetilde M$ has no $PG(m-1,\bFp)$-minor, 
then either
$M$ is a member of at least one of the classes
$\cM(\Phi_1),\ldots,\cM(\Phi_s)$, or
$M^*$ is a member of at least one of the classes
$\cM(\Psi_1),\ldots,\cM(\Psi_t)$.
\end{itemize}
\end{theorem}

\section{Girth Conjectures}

We now turn our attention to conjectures for which we believe that the previous theorems may 
be of some use. We begin with conjectures related to the girth of matroids.

Recall that the {\em girth} of a matroid is the size of its smallest circuit. 
The problem of determining the distance of a binary linear code is equivalent to 
finding the girth of an associated binary matroid.  The connection with coding
theory guaranteed that the problem of determining the girth of a binary matroid
attracted considerable attention; the problem was eventually
shown to be NP-hard by Vardy \cite{vardy}.

On the other hand finding the girth of a graph is easily seen to 
reduce to a shortest-path problem
and can hence be solved in polynomial time.  Finding the girth of a 
cographic matroid is the problem of finding a minimum cutset in the graph which
is also polynomial-time solvable.  The next conjecture follows 
the line of thought that if a property holds for both the class of graphic matroids
and the class of cographic matroids, then it may well extend to all proper minor-closed
classes of binary matroids.

\begin{conjecture}
\label{girth1}
Let $\mathcal M$ be a proper minor-closed class of binary matroids. Then there is 
a polynomial-time algorithm that, given as input a member $M$ of $\mathcal M$, 
determines the girth of $M$.
\end{conjecture}

We have strong evidence that Conjecture~\ref{girth1} is true.
There is a natural extension of Conjecture~\ref{girth1}
to other fields,
although we have neither the evidence for, nor much faith in,
the following generalisation.

\begin{conjecture}
\label{girth2}
Let $\mathbb F$ be a finite field and let $\mathcal M$ be a minor-closed class of 
$\mathbb F$-represented matroids that does not contain all matroids over 
$\bFp$. Then there is a polynomial-time algorithm that,
given as input a generator matrix for a member $M$ of $\cM$,
determines the girth of $M$.
\end{conjecture}

Now there is no reason to suspect that the above conjectures should
reduce to highly connected instances, but the structure theorems
do suggest the following step towards Conjecture~\ref{girth1}.
\begin{conjecture}
Let $t\in\bZ_+$. There is a polynomial-time algorithm that,
given as input three binary matrices $A,B,P$ such that
$B$ is the incidence matrix of a graph, $\rank(P)\le t$, and $A=B+P$,
determines the girth of $M(A)$.
\end{conjecture}
Naturally Conjecture~\ref{girth2} gives rise to a similar conjecture
on perturbations of frame matroids.

The {\em cogirth} of a matroid is the size of its smallest cocircuit.
It follows from a seminal theorem of Mader \cite{mader} that a simple graph with sufficiently
large cogirth has a $K_n$-minor. Thomassen \cite{thomassen} proved the related result
that a cosimple graph with sufficiently large girth has a  $K_n$-minor. Combining these facts
and generalising we obtain the next conjecture.

\begin{conjecture}
\label{girth3}
Let $\mathbb F$ be a finite field and let $n$ be a positive integer.
Then there is an integer $c$ such that
each $\mathbb F$-representable matroid with no $M(K_n)$- or 
$M(K_n)^*$-minor has girth at most $c$.
\end{conjecture}

The maximum girth of a simple $n$-vertex graph with no series pairs grows logarithmically in
$n$. On the other hand, the girth of $M(K_n)^*$ is $n-1$ and
its rank is only $\frac{(n-1)(n-2)}{2}$.
We expect that excluding $M(K_n)^*$ will bring us back
to the behaviour we see in graphs.

\begin{conjecture}
\label{girth4}
Let $\mathbb F$ be a finite field and let $n$ be a positive integer.
Then there is an integer $c$ such that
each cosimple
$\mathbb F$-representable matroid with no $M(K_n)^*$-minor, has girth
at most $c\cdot \log(r(M))$.
\end{conjecture} 
 
\section{Conjectures in Coding Theory}

Interest in the problem of finding the girth of a binary matroid was primarily due to 
the connection with coding theory.
In this section we consider some other problems that connect
coding theory with matroids.
The connection is not particularly surprising since
a {\em linear code} over a finite field $\bF$ and 
an $\bF$-represented matroid  $M=(E,U)$ are one and the same thing;
the elements of $U$ are the {\em code words}.
There is even some common terminology,
the {\em dual code} of $M$ is $M^*$.

However, there are also a number of distinctions in terminology, for example,
coding theorists typically require that their generator matrices 
have linearly independent rows.
Also, in coding theory, the terms {\em puncturing}
and {\em shortening} are used in place of deletion and contraction.
The {\em distance} (or {\em Hamming distance}) of the
code $M$, which we will denote by $d(M)$, is the cogirth of the matroid $M$;
the {\em length} is $|M|$; and the {\em dimension} is $r(M)$.
A more uncomfortable difference in terminology is the
switch between graphic and cographic;
the code $M$ is {\em graphic} if the matroid $M$ is cographic and
the code $M$ is {\em cographic} if the matroid $M^*$ is graphic.

\subsection*{Asymptotically good families of codes}
Two competing measures of the quality of a code $M$ are its
{\em rate}, which is defined
as $r(M)/|M|$, and its {\em relative distance}, which is defined as $d(M)/|M|$.
The rate measures efficiency whereas the relative distance is a coarse
measurement of the tolerance of the code to errors.
A family $\mathcal C$ of codes is 
{\em asymptotically good} if there exists a real number $\epsilon>0$ and an infinite 
sequence $C_1,C_2,\ldots$ of codes in $\mathcal C$ with increasing dimension such that
$$ \frac{r(C_i)}{|C_i|}\ge \epsilon \mbox{ and } \frac{d(C_i)}{|C_i|} \ge \epsilon$$
for each $i$.
It is straightforward to prove that the class of binary linear codes
is asymptotically good (see~\cite[Chapter 9]{vanlint}).
On the other hand, Kashyap \cite{navin1}
proved that the class of graphic codes is not asymptotically good.
Later, in an unpublished note \cite{navin2}, 
he extended this to the class of {\em regular codes}, these are the binary linear codes 
that correspond to regular matroids.
We expect the behaviour of regular codes to be typical
for proper minor-closed classes of binary linear codes.

\begin{conjecture}
\label{code1}
No proper minor-closed class of binary linear codes is asymptotically good.
\end{conjecture}

We also expect the behaviour of linear codes over other fields to be similar.

\begin{conjecture}
\label{code2}
Let $\mathcal C$ be a minor-closed class of linear codes over a finite field $\mathbb F$.
If $\mathcal C$ does not contain
all linear codes over $\bFp$, then $\mathcal C$ is not asymptotically good.
\end{conjecture}

\subsection*{Threshold functions}
While graphic codes are not asymptotically good, they are certainly better than cographic
codes. What follows is a means of making such comparisons possible.
We will restrict our discussion here to binary linear codes.

For each class $\cC$ of binary linear codes there exists a
function $\theta_{\cC}(R):(0,1)\rightarrow [0,1]$ such that,
for any real number $R\in (0,1)$, the following hold:
\begin{itemize}
\item[(i)] if the bit-error-probability $p$ of the channel is smaller
than $\theta_{\cC}(R)$,
then for each $\epsilon>0$ there exists
a code of rate greater than $R$ in $\mathcal C$
for which the probability of error
using maximum-likelihood decoding is less than $\epsilon$,
and
\item[(ii)] 
if the bit-error-probability $p$ of the channel is greater than $\theta_{\cC}(R)$,
then there exists an $\epsilon>0$,
such that for all codes of rate less than $R$ in $\mathcal C$
the probability of error using 
maximum-likelihood decoding is greater than $\epsilon$.
\end{itemize}
In other words, for any $p<\theta_{\cC}(R)$, the probability of error goes to $0$ for 
appropriately chosen codes from $\mathcal C$ of rate $R$ 
with increasing dimension, while if $p>\theta_\cC(R)$, the probability of error is bounded
away from $0$.

The function $\theta_{\cC}$ is called the {\em threshold function}
(or {\em ML threshold function}) for $\cC$.  The threshold function
is known explicitly for the class of binary linear codes and for the
class of graphic codes.  For the class $\cB$ of binary linear codes,
it follows from Shannon's Theorem that
$$ \theta_{\cB}(R) = f^{-1}(R), $$
where $f:(0,\frac 1 2)\rightarrow [0,1]$ is the invertible function defined by
$$ f(p) = 1 + p\log_2(p) + (1-p)\log_2(1-p). $$
For the class $\cG$ of graphic codes, Decreusefond and Zemor \cite{dec}
proved\footnote{Note added in print:
$\theta_{\cG}(R)=\frac{1}{2}(1 - \sqrt R)^2(1+R)^{-1}$ should be read as a conjecture.
As pointed out by Peter Nelson and Stefan van Zwam (personal communication),
\cite{dec} only considers classes of regular graphs.}
that
$$ \theta_{\cG}(R) = \frac{(1 - \sqrt R)^2}{2(1+R)}.$$

While graphic codes are not as good as arbitrary linear codes, they are still reasonable,
perhaps surprisingly so. On the other hand, it is easy to show that 
the class of cographic codes is very poor.
\begin{lemma}
For the class $\mathcal G^*$ of all cographic codes, $\theta_{{\mathcal G}^*}(R)=0$ for all
$R\in(0,1)$.
\end{lemma}

\begin{proof}
Let $R\in (0,1)$ be a real number,
let $\delta = \frac{1}{2(1-R)}$, and let $\epsilon = p^\delta$.
Consider any cographic code $C$ with rate at least $R$.
Now $C$ is the cycle matroid of a connected graph $G=(V,E)$.
The fact that $C$ has rate at least $R$ can be expressed graphically as
$$ \frac{ |E| - |V| +1}{|E|}  \ge R,$$
which implies that
$$ |E| <\frac{|V|}{1-R}=\tfrac{1}{2}\delta |V|.$$
Therefore $G$ has a vertex of degree at most $\delta$, which means 
that $C$ has distance $\le \delta$.
Then the error-probability is at least $\epsilon$.
\end{proof}

This brings us to a striking conjecture.

\begin{conjecture}
\label{code3}
Let $\mathcal C$ be a proper minor-closed class of binary linear codes.
\begin{itemize}
\item[(i)] If $\mathcal C$ contains all graphic codes, then 
$\theta_{\mathcal C}=\theta_{\mathcal G}$.
\item[(ii)] If $\mathcal C$ does not contain all graphic codes, then 
$\theta_{\mathcal C}=0$.
\end{itemize}
\end{conjecture}

Let $\bF$ be a finite field of order $q$.
Threshold functions can be defined analogously for codes over $q$-ary symmetric channels,
we omit the obvious definitions.
We expect that something like Conjecture~\ref{code3} holds for
arbitrary finite fields; in particular, we conjecture that 
there  is a finite list $(\cC_1,\ldots,\cC_k)$ of ``special"
families of codes such that, if 
$\cC$ is a minor-closed class of linear codes over $\bF$,
then its threshold function will
be the maximum of the threshold functions of those families
among $\cC_1,\ldots,\cC_k$ that are contained in $\cC$.
The special families will include the classes of linear codes over the
subfields of $\bF$, as well as the classes
$\cD(\bF,\Gamma)^*$ where $\Gamma$ is a subgroup of $\bFmult$.
In fact, we believe that the 
aforementioned families define $(\cC_1,\ldots,\cC_k)$.
Each of these families contains the class of graphic codes, which gives rise 
to the following conjecture. 
\begin{conjecture}
\label{code4}
Let $\mathbb F$ be a finite field and let $\mathcal C$ be a proper minor-closed class of
linear codes over $\mathbb F$.
If $\mathcal C$ does not contain all graphic codes, then $\theta_{\mathcal C}=0$.
\end{conjecture}

By using Shannon's Theorem, we can determine the threshold function
for the class of linear codes over any given subfield of $\bF$.
However, the following problem is open:
\begin{problem}
Let $\bF$ be a finite field and let $\Gamma$ be a
subgroup of $\bFmult$.
Determine the threshold function for $\cD(\bF,\Gamma)^*$.
\end{problem}

The conjectures in this section should all be approachable by using Theorem~\ref{structure1};
due to the nature of the conjectures, the refined structure theorems in Section 4
should not be required.

\section{Growth Rates}

The {\em growth rate} of a class $\mathcal M$ of matroids is the function $h_{\cM}$,
where $h_{\cM}(r)$
is the maximum number of elements in a simple rank-$r$ member of $\mathcal M$, if that
maximum exists, otherwise we say that the growth rate is infinite.

In this section we 
describe, via some results and some conjectures, what we believe to be
the fundamental mechanisms governing growth rates of minor-closed classes.
In Section~\ref{growthoverfinitefields} we pose more specific conjectures
about the growth rate of minor-closed classes of matroids over finite fields.

Kung \cite{kung} conjectured that growth rates of minor-closed classes of matroids 
are either linear, quadratic, exponential, or infinite.
Kung's conjecture was eventually proved as the culmination of 
results in a sequence of papers \cite{gk1,gkw1,gw2}.

\begin{theorem}
\label{growth-rates}
Let $\mathcal M$ be a minor-closed class of matroids. Then either
\begin{itemize}
\item[(i)] $h_{\cM}(r) = O(r)$,
\item[(ii)] $\mathcal M$ contains all graphic matroids and 
$h_{\mathcal M}(r)  = O(r^2)$, 
\item[(iii)] there is a finite field $\bF$ of order $q$
such that $\mathcal M$ contains all $\bF$-representable matroids and 
$h_{\mathcal M}(r) = O(q^r)$, or
\item[(iv)] $\mathcal M$ contains all simple rank-$2$ matroids.
\end{itemize}
\end{theorem}

We say that a simple  rank-$r$ matroid $M \in\mathcal M$ 
is {\em extremal} if $|M|=h_{\mathcal M}(r)$.
We say that a minor-closed class $\cM$ has {\em linear density}
if outcome (i) of Theorem~\ref{growth-rates} holds.
If outcome (ii) holds, we say that $\mathcal M$ is 
{\em quadratically} dense, and
if outcome (iii) holds, we say that $\mathcal M$ is 
{\em base-$q$ exponentially} dense.

Given functions $f,g:\bZ_+\rightarrow \bZ$ we will write
$f(r)\approx g(r)$ to denote that $f(r) = g(r)$ for all
sufficiently large $r$; we say that $f(r)$ and $g(r)$ are 
{\em eventually equal}.

\subsection*{Exponentially dense classes}
Exponentially dense classes are the easiest to understand in terms of growth rates;
this is due, in part, to the fact that the extremal matroids are very highly connected.
Geelen and Nelson \cite{nel} proved the following refinement of Theorem~\ref{growth-rates}.
 
\begin{theorem}
\label{peters-theorem}
Let $q$ be a prime power and $\mathcal M$ be a base-$q$ exponentially dense
minor-closed class of matroids. Then there exist $k,d\in \bZ_+$ with
$0\leq d\leq \frac{q^{2k}-1}{q^2-1}$, such that
$$h_{\mathcal M}(r)\approx \frac{q^{r+k}-1}{q-1}-qd.$$
\end{theorem}

Geelen and Nelson prove a little more;
they show that the growth rate function is attained 
by rank-$k$ projections of projective geometries.

\subsection*{Quadratically dense classes}
We know considerably less about classes of quadratic density,
although this looks to be a promising direction for future research.
The next conjecture, which may not be difficult to prove,
is that the extremal matroids are also highly connected in this case.

\begin{conjecture}
\label{growth1}
Let $\mathcal M$ be a quadratically dense minor-closed class. Then, for each $k\in\mathbb Z_+$,
there is an integer $r$ such that all extremal matroids in $\mathcal M$ with rank at 
least $r$ are vertically $k$-connected.
\end{conjecture}

Next we conjecture that growth rates of quadratically dense classes
are eventually quadratic functions.

\begin{conjecture}
\label{growth2}
Let $\mathcal M$ be a quadratically-dense minor-closed class of matroids. Then
there is a quadratic polynomial $p$ such that, $h_{\mathcal M}(r)\approx p(r)$.
\end{conjecture}

We expect extremal matroids in classes of 
quadratic growth rates to be perturbed frame matroids.
In essence the next conjecture 
says that lifts of frame matroids determine the leading
coefficient in quadratic growth-rate functions. First we give some terminology.

Call a matroid $M$ an $(\alpha,t)$-{\em frame
matroid} if it has a basis $V\cup T$ with $|T|= t$ such that
\begin{itemize}
\item[(i)] the fundamental circuit of any $e\in E(M) - (V\cup T)$ contains
at most two elements of $V$, and
\item[(ii)] for each $u,v\in V$ there are $\alpha$ elements
that are in the span of $T\cup\{u,v\}$ but not in the span of either
$T\cup\{u\}$ or $T\cup \{v\}$.
\end{itemize}

\begin{conjecture} 
\label{growth2b}
Let $\cM$ be a  quadratically dense minor-closed class of 
matroids. Then there exist $\alpha, t\in\bZ_+$ such that
\begin{itemize}
\item
$h_{\cM}(r) =  \alpha{r\choose 2} + O(r)$, and
\item
for each integer $r\ge t$, 
$\cM$ contains an $(\alpha,t)$-frame matroid of rank $r$.
\end{itemize}
\end{conjecture}

\subsection*{Linearly dense classes}
Extremal members of linearly dense classes are not 
always highly connected. This may make it more difficult
to understand the growth rates of these classes.
Even for graphic matroids this is still not well understood;
see, for example, the problems listed by Eppstein \cite{epp}.
Sergey Norin posed conjectures that, if true, shed light 
on growth rates of minor-closed classes of graphs.
The following three conjectures extend those conjectures
of Norin to matroids.

\begin{conjecture}
\label{linear1}
Let $\mathcal M$ be a linearly dense minor-closed class of matroids. 
Then there exists a sequence $(a,b_0,b_1,\ldots,b_{t-1})$ of 
integers such that, for all sufficiently large $r$
we have $h_{\mathcal M}(r)=ar+b_i$ where
$i\in\{0,\ldots,t-1\}$ and $i\equiv r (\mbox{mod } t)$.
\end{conjecture}

\begin{conjecture}
\label{linear2}
Let $\mathcal M$ be a linearly dense minor-closed class of matroids. Then
$\lim_{r\to\infty}h_{\cM}(r)/r$ exists and is rational.
\end{conjecture}

\begin{conjecture}
\label{linear3}
Let $\mathcal M$ be a linearly-dense minor-closed class of matroids. Then
$\lim_{r\to\infty}h_{\cM}(r)/r$ exists and is achieved by a subfamily of $\cM$
of bounded pathwidth.
\end{conjecture}

\section{Growth rates for classes over finite fields}\label{growthoverfinitefields}

The extremal members of exponentially dense minor-closed classes
are known to be highly connected and the extremal members of
quadratically dense classes are conjectured to be highly connected.
Therefore, Theorems~\ref{refinement1} and~\ref{refinement2}
should explain the specific mechanisms that control growth rates
of exponentially dense and quadratically dense minor closed
classes of matroids represented over a given finite field.
With a bit of additional work, one might be able to extract information
about the eventual growth rates for particular classes of interest.

\subsection*{Representation over two fields}
First consider the class of matroids that are representable
over two given fields $\bF_1$ and $\bF_2$.
One of the fields needs to be finite or the growth rate will be infinite.

\begin{problem}
\label{twofields}
Let $\bF_1$ and $\bF_2$ be fields with $\bF_1$ finite and let
$\cM$ be the class of matroids representable over both $\bF_1$ and $\bF_2$.
Determine the growth-rate function for $\cM$.
\end{problem}

Explicit answers are known for this problem in the
case that $|\bF_1|= 2$, see~\cite{heller}, and in the case that
$|\bF_1|=3$, see~\cite{kung3,ko,ovw}.

If $\bF_1$ and $\bF_2$ have the same characteristic,
then the class $\cM$ will be base-$q$ exponentially dense,
where $q$ is the size of the largest common subfield, up to isomorphism,
of $\bF_1$ and $\bF_2$.
We hope that the eventual growth rate functions will
be completely determined for all instances of Problem~\ref{twofields};
the most general partial result is due to 
Nelson~\cite{nelson2}.
\begin{theorem}
Let $q$ be a prime power, let $j\ge 3$ be an odd integer, and let 
$\cM$ be the class of matroids representable
over the field of order $q^2$ as well as over the field of order $q^j$.
Then
$$ h_{\cM}(r) \approx \frac{q^{r+1}-1}{q-1} -q. $$
\end{theorem}

If $\bF_1$ and $\bF_2$ have different characteristics,
then $\cM$ contains all graphic matroids but no projective planes
and, hence, $\cM$ is quadratically dense.
For this case, Kung~\cite{kung} proved quite good bounds on the
growth rate function 
and posed some interesting conjectures which we will expand upon below.
In essence we believe that the extremal matroids 
are projections of frame matroids; no lifts are required.
Note that, if $\Gamma$ is a subgroup of $\bFmult$ and $\alpha=|\Gamma|$, then
$$ 
h_{\cD(\bF,\Gamma)}(r) = \alpha {r\choose 2} + r.
$$
Moreover, note that projections only affect the linear term in the 
growth rate function since $h(r+1) = h(r) + O(r)$ for any
quadratic function $h$.
\begin{conjecture}
Let $\bF_1$ be a finite field, let $\bF_2$ be a
field with different characteristic from $\bF_1$,
let $\alpha$ be the size of the largest common
subgroup, up to isomorphism, of the
groups $\bFmult_1$ and $\bFmult_2$,
and let $\cM$ be the class of matroids representable over 
both $\bF_1$ and $\bF_2$.  Then 
$$h_{\cM}(r) =  \alpha {r \choose 2} + O(r).$$
\end{conjecture}

When $\bFmult_1$ is a subgroup of $\bFmult_2$, we expect to do even better.
\begin{conjecture}\label{part1}
Let $\bF_1$ be a finite field of order $q$, let $\bF_2$ be a
field with different characteristic from $\bF_1$,
and let $\cM$ be the class of matroids representable over 
both $\bF_1$ and $\bF_2$.
If $\bFmult_1$ is a subgroup of $\bFmult_2$, then 
$$h_{\cM}(r) \approx (q-1) {r \choose 2} + r.$$
\end{conjecture}

\subsection*{Excluding a minor}
Next we consider classes obtained by excluding a single minor $N$.
\begin{problem}
\label{excludedminor}
Let $\bF$ be a finite field, let $N$ be a matroid,
and let $\cM$ be the class of 
$\bF$-representable matroids with no $N$-minor.
Determine the eventual growth-rate function for $\cM$.
\end{problem}

If $N$ is not representable over $\bFp$, 
then $\cM$ will be base-$q$ exponentially dense, where
$q$ is the size of the largest subfield of $\bF$ over which
$N$ is not representable.
Now, by Theorem~\ref{peters-theorem}, there exist $d,k\in\bZ_+$ such that
$$ h_{\cM}(r) \approx \frac{q^{r+k}-1}{q-1} - qd.$$
At the very least, Theorem~\ref{refinement1}
should give an algorithm for computing $k$ and $d$.

Problem~\ref{excludedminor} remains open even in the
benign-looking case that $N$ is a line.
We hope, however, that the eventual growth rate functions will be determined explicitly 
for quite general instances of Problem~\ref{excludedminor},
such as when $N$ is a projective geometry or an affine geometry.
Nelson~\cite{nelson2} has proved one such result along these lines.
\begin{theorem}
Let $\bF$ be a finite field of square order $q^2$, let $n\ge 3$,
and let $\cM$ be
the set of $\bF$-representable matroids
with no $\PG(n+1,\bF)$-minor. Then
$$ h_{\cM}(r) \approx \frac{q^{r+n}-1}{q-1} - q\frac{q^{2n}-1}{q^2-1}.$$
\end{theorem}

Next consider the case that 
$N$ is representable over $\bFp$
but that $N$ is not graphic.
In this case $\cM$ will be quadratically dense.

Let $L_1$, $L_2$ and $L_3$ be three lines through
a point $e$ in the projective plane $\PG(2,\bF)$,
let $a$ and $b$ be distinct points $L_2-\{e\}$,
and let $M$ be the restriction of $\PG(2,\bF)$
to $L_1\cup L_3\cup\{a,b\}$.
The matroid $M$ is independent of the particular
choice of $(L_1,L_3,a,b)$; we denote $M$ by $\cR(\bF)$.
These matroids are called {\em Reid geometries}
and play a significant role in~\cite{kung}.
The Reid geometry $\cR(\bF)$ is not representable
over any field whose characteristic is different
from that of $\bF$.

The following conjecture arose from discussions with Joseph Kung;
it generalises Conjecture~\ref{part1}.
\begin{conjecture}
Let $\bF$ be a finite field of order $q$
and let $\cM$ be the class of $\bF$-representable
matroids with no $\cR(\bFp)$-minor.  Then
$$ h_{\cM}(r) \approx (q-1) {r\choose 2} + r.$$
\end{conjecture}

\section{Beyond Finite Fields}

When we go from minor-closed classes of matroids
representable over finite fields to
arbitrary minor-closed classes
the nice properties quickly slip away; for example, the
set of all matroids with rank at most $3$ is not well-quasi-ordered.
However, it looks like minor-closed classes that do not
contain all uniform matroids remain ``highly structured";
this is discussed in the survey paper~\cite{uniform}.

Our next conjecture, if true, is an extension of our structure theorem, Theorem~\ref{structure1},
to minor-closed classes of matroids that omit a certain uniform matroid.
To state that conjecture we need general matroidal analogues
of perturbation and of represented frame matroids.

Let $M_1$ and $M_2$ be matroids with a common ground set, say $E$.
If there is a matroid $M$ on ground set $E\cup \{e\}$ such that
$M_1 = M\del e$ and $M_2=M\con e$, then we say that $M_2$ is an {\em elementary
projection} of $M_1$ and that $M_1$ is an {\em  elementary lift} of $M_2$. 
We let $\dist(M_1,M_2)$ denote the minimum number of elementary lifts and
elementary projections required in order to transform $M_1$ into $M_2$.

The matroid $M$ is a {\em frame matroid} if there exists a matroid $M'$
with a basis $B$ such that $M=M'\del B$,
and every element of $E(M)$ is spanned by at most two elements of $B$ in $M'$.
Zaslavsky \cite{zas} has shown that frame matroids can be canonically associated
with the so-called bias graphs; see \cite{zas} or Oxley \cite[Chapter~6.10]{oxley} for details.

\begin{conjecture}
\label{beyond1}
Let $U$ be a uniform matroid.
Then there exist $k,t,q\in \bZ_+$ such that,
if $M$ is a vertically $k$-connected matroid that has no $U$-minor,
then there exists a matroid $N$ with $\dist(M,N)\leq t$
such that either
\begin{itemize}
\item[(i)] $N$ is a frame matroid,
\item[(ii)] $N^*$ is a frame matroid, or 
\item[(iii)] $N$ is representable over a finite field of size at most $q$.
\end{itemize}
\end{conjecture}

Conjecture~\ref{beyond1} is likely to be difficult; it would be a significant
step to prove the result for representable
matroids with no $U_{2,n}$-minor and no $U_{n-2,n}$-minor.

In order to prove Conjecture~\ref{beyond1},
it would be useful to have the following generalisation of Theorem~\ref{structure3}.
We assume that the reader is familiar with the so-called bicircular matroid
BM$(G)$ of a graph $G$; this is a particular type of frame matroid; see~\cite[Chapter 6.10]{oxley}.

\begin{conjecture}
\label{beyond2}
Let $n$ be a positive integer.
Then there exists $k\in\bZ_+$ such that
each vertically $k$-connected matroid has a
$U_{n,2n}$-, $M(K_n)$-,
$M(K_n)^*$-, $BM(K_n)$- or $BM(K_n)^*$-minor.
\end{conjecture}

Again the special case where $M$ is representable
with no $U_{2,n}$-minor and no $U_{n-2,n}$-minor is of considerable interest.

\section*{Acknowledgements} We thank Navin Kashyap for helpful discussions on 
the problems in coding theory. We thank Sergey Norin for clarifying the status of
conjectures on extremal members and growth rates of minor-closed classes of graphs.

\end{document}